\documentclass{article}
\usepackage[margin=1in]{geometry}
\usepackage{mathtools}
\usepackage{amssymb}
\usepackage{amsthm}
\usepackage{dsfont}
\usepackage[unicode, psdextra]{hyperref}
\usepackage{comment}
\hypersetup{
    colorlinks=true,
    linkcolor=blue,
    citecolor=blue,
    urlcolor=blue
}
\usepackage{authblk}

\theoremstyle{plain}
\newtheorem{theorem}{Theorem}

\newtheorem{corollary}[theorem]{Corollary}

\theoremstyle{definition}

\newtheorem{definition}[theorem]{Definition}
\newtheorem{remark}[theorem]{Remark}

\DeclareMathOperator{\sgn}{sgn}

\newcommand{\E}{\mathbb{E}}
\newcommand{\R}{\mathbb{R}}

\DeclarePairedDelimiter{\norm}{\lVert}{\rVert}
\newcommand{\opnorm}[2][]{\norm[#1]{#2}_{\text{op}}}

\DeclarePairedDelimiterX{\inner}[2]{\langle}{\rangle}{#1, #2}

\title{\textbf{On the Subgaussianity of Quantized Linear Maps: \\
An AI-Assisted Note}}

\author[1]{Guangyi Zou}
\author[1]{Roman Vershynin}

\affil[1]{Department of Mathematics, University of California, Irvine
\protect\\ \small \texttt{gzou3@uci.edu}, \texttt{rvershyn@uci.edu}}

\date{\today}

\begin{document}

\maketitle

\begin{abstract}
We prove an elementary bounded-differences inequality for functions of non-isotropic Gaussian vectors. Specifically, if $f$ has bounded coordinate differences and $X\sim\mathcal N(\mu,\Sigma)$, then the resulting concentration bound depends on the condition number $\kappa(\Sigma)$. As an application, we answer a question of Simone Bombari concerning the subgaussianity of sign-quantized linear maps $Y=\sgn(Wx)$.

In the special case where $f$ is the coordinatewise sign function, an argument was initially suggested to us by Gemini 3.5 Flash without attribution. We subsequently discovered that it closely resembles an earlier argument of Barber and Kolar~\cite[Lemma~4.5]{MR3852657}. This revision corrects the attribution and documents the episode as an instance of AI-assisted mathematical discovery.
\end{abstract}

\vspace{1em}

\section{Introduction: An AI-for-Math Perspective}

Simone Bombari asked us whether the $1$-bit quantized random vector
$Y = \sgn(Wx)$ has subgaussian norm bounded by a universal constant. Here
$W$ is an $n \times n$ random Gaussian matrix, and $x$ is an independent
standard normal random vector in $\mathbb{R}^n$. The question is nontrivial
since the coordinates of $Y$ are not independent. We give a positive answer
and, more generally, prove a bounded-differences inequality for a
non-isotropic Gaussian input.

\begin{itemize}
    \item \textbf{AI (re)discovery
    (Theorem~\ref{thm:bounded_differences}):}
    During our initial exploration, Gemini 3.5 Flash\footnote{The original
    Chinese dialogue is archived at
    \url{https://gemini.google.com/share/daca36d5cfbc}; see
    \url{https://github.com/The8ravo/conversation-history-with-Gemini/blob/main/conversation_history_with_gemini.pdf}
    for an AI-generated English translation with a timeline. The reproduced
    interactive dialogue and full derivations are archived at
    \url{https://gemini.google.com/share/cefbb71da6d2}.}
    proposed decomposing the Gaussian vector into independent parts, using one part to smooth the sign function, and then applying Gaussian concentration. We later learned this idea had already appeared in Barber and Kolar~\cite[Lemma~4.5]{MR3852657}. Gemini later returned essentially their same argument without citing the source. The present revision corrects this attribution. The argument extends to general bounded-differences
    functions, as stated below.

    \item \textbf{Human design (Section~\ref{sec:application}):}
    To apply this result to square random matrices $W$, which are typically
    ill conditioned, the authors used a row-partitioning trick. This breaks
    the square matrix into two well-conditioned rectangular blocks.
\end{itemize}

\section{Main Result: Gaussian Bounded Differences}

The subgaussian norm of a random vector is defined as the largest subgaussian
norm of its $1$-dimensional marginals:
\[
\|Y\|_{\psi_2}
= \sup_{v \in S^{n-1}} \|\langle v, Y \rangle\|_{\psi_2},
\]
see~\cite[Definition 3.4.1]{Vershynin_2026}. The condition number of a
positive-definite matrix $\Sigma$ is the ratio of its largest eigenvalue to its smallest eigenvalue:
\[
\kappa(\Sigma)
\coloneqq \frac{\lambda_{\max}(\Sigma)}{\lambda_{\min}(\Sigma)}.
\]

\begin{definition}[Bounded differences]
A measurable function $f:\R^n\to\R$ has bounded-differences constants
$b_1,\ldots,b_n\ge 0$ if, for every $i=1,\ldots,n$, changing only the $i$-th coordinate changes $f$ by at most $b_i$:
\[
|f(x)-f(x')|\le b_i
\quad \text{whenever} \quad x_j=x'_j \text{ for all } j\ne i.
\]
Write $b=(b_1,\ldots,b_n)$.
\end{definition}

\begin{theorem}[Gaussian bounded-differences inequality]
\label{thm:bounded_differences}
Let $X\sim\mathcal N(\mu,\Sigma)$ be an $n$-dimensional Gaussian vector with non-singular covariance matrix $\Sigma$. Let $f:\R^n\to\R$ have
bounded-differences constants $b=(b_1,\ldots,b_n)$. Then,
\[
\|f(X) - \E f(X)\|_{\psi_2} \le C \sqrt{\kappa(\Sigma)}\|b\|_2,
\]
where $C>0$ is an absolute constant.
\end{theorem}

For comparison, if $X$ is any random vector with independent coordinates, the classical
bounded-differences inequality gives
\[
\|f(X)-\E f(X)\|_{\psi_2}\le C\|b\|_2,
\]
where $C>0$ is an absolute constant. Thus,
Theorem~\ref{thm:bounded_differences} extends this conclusion to dependent
Gaussian coordinates at the cost of the factor $\sqrt{\kappa(\Sigma)}$.

\begin{proof}
The proof extends the argument of Barber and Kolar~\cite[Lemma~4.5]{MR3852657}. Replacing $f$ by $f-f(0)$ does not change
either its bounded-differences constants or the centered random variable $f(X)-\E f(X)$. After this replacement, a coordinate-by-coordinate telescoping argument gives $|f(x)|\le\sum_i b_i$ for every $x\in\R^n$.
Thus $f$ is bounded, which justifies the Gaussian convolution and differentiation used below.

\medskip
\noindent\emph{Step 1: Covariance splitting.}
Let $\alpha=\lambda_{\min}(\Sigma)$ and $\beta=\lambda_{\max}(\Sigma)$. Define
$\Sigma_G=\Sigma-\alpha I_n\ge 0$. We may write
\[
X=\mu+\sqrt \alpha\,Z+G,
\]
where $Z\sim\mathcal N(0,I_n)$ and
$G\sim\mathcal N(0,\Sigma_G)$ are independent.

\medskip
\noindent\emph{Step 2: Conditional bounded-differences estimate.}
Define the function
\[
H(g)=\E_Z f(\mu+\sqrt \alpha\,Z+g), \quad g \in \R^n.
\]
Conditionally on $G$, the coordinates of $Z$ are independent, and changing
$Z_i$ changes only the $i$-th coordinate of the input of $f$. Thus the map
$z\mapsto f(\mu+\sqrt\alpha\,z+G)$ has bounded-differences constants
$b_1,\ldots,b_n$. By the classical bounded-differences inequality
(see, e.g.,~\cite[Theorem~5.7.1]{Vershynin_2026}), conditionally on $G$,
\[
\big\|f(\mu+\sqrt\alpha\,Z+G)-H(G)\big\|_{\psi_2\mid G}
\le C\|b\|_2,
\]
where $C>0$ is an absolute constant. More precisely, its
exponential-moment form gives, for every $\lambda\in\R$,
\begin{equation}\label{eq:conditional_bounded_differences}
\E\!\left[
\exp\!\left(\lambda
\big(f(\mu+\sqrt \alpha\,Z+G)-H(G)\big)\right)
\,\middle|\,G
\right]
\le \exp\!\left(\frac{\lambda^2\|b\|_2^2}{8}\right).
\end{equation}

\medskip
\noindent\emph{Step 3: Concentration of the conditional mean.}
Gaussian convolution makes $H$ smooth: Gaussian integration by parts
(see, e.g.,~\cite[Lemma~7.2.3]{Vershynin_2026}) gives
\begin{equation}
    \label{eq:partial}
    \partial_iH(g) = \frac{1}{\sqrt \alpha}\E_Z\!\left[Z_i
f(\mu+\sqrt \alpha\,Z+g)\right].
\end{equation}
Conditionally on $Z_j$ for all $j\ne i$, the range of
$f(\mu+\sqrt \alpha\,Z+g)$ as a function of $Z_i$ has length at most $b_i$.
Subtracting the midpoint of this range does not change the expectation in
\eqref{eq:partial}. This gives
\[
|\partial_iH(g)|
\le \frac{b_i}{2\sqrt \alpha}\E|Z_i|
=\frac{b_i}{\sqrt{2\pi \alpha}}.
\]
Thus $\|\nabla H(g)\|_2\le\|b\|_2/\sqrt{2\pi\alpha}$, so $H$ is a
$\|b\|_2/\sqrt{2\pi\alpha}$-Lipschitz function.

Let $U\sim\mathcal N(0,I_n)$. Since
$G\stackrel{d}{=}\Sigma_G^{1/2}U$, define
\[
F(u)=H(\Sigma_G^{1/2}u).
\]
The Lipschitz constant of $F$ satisfies
\[
\operatorname{Lip}(F)\le \opnorm{\Sigma_G^{1/2}}\operatorname{Lip}(H)\le\sqrt{\opnorm{\Sigma_G}}\,
\frac{\|b\|_2}{\sqrt{2\pi\alpha}}=\sqrt{\beta-\alpha}\,
\frac{\|b\|_2}{\sqrt{2\pi\alpha}}.
\]
Since $\beta/\alpha=\kappa(\Sigma)$, it follows that
\[
\operatorname{Lip}(F)^2
\le \frac{\kappa(\Sigma)-1}{2\pi}\|b\|_2^2.
\]
Standard Gaussian concentration applied to $F(U)$
(see, e.g.,~\cite[Theorem~5.2.3]{Vershynin_2026} or
\cite[Theorem~5.5]{MR3185193}) therefore gives
\begin{align}
\E\exp\!\left(\lambda(H(G)-\E H(G))\right)
&=\E\exp\!\left(\lambda(F(U)-\E F(U))\right)\notag\\
&\le \exp\!\left(\frac{\lambda^2\operatorname{Lip}(F)^2}{2}\right)\notag\\
&\le
\exp\!\left(
\frac{\lambda^2\|b\|_2^2}{4\pi}
(\kappa(\Sigma)-1)
\right).
\label{eq:mean_concentration}
\end{align}

\medskip
\noindent\emph{Step 4: Combining the bounds.}
Conditioning on $G$ and using
\eqref{eq:conditional_bounded_differences}--\eqref{eq:mean_concentration}, we obtain
\[
\log\E\exp\!\left(\lambda(f(X)-\E f(X))\right)\le \lambda^2\|b\|_2^2
\left(\frac18+\frac{\kappa(\Sigma)-1}{4\pi}\right)\le \frac{\lambda^2\kappa(\Sigma)\|b\|_2^2}{8}.
\]
By the standard equivalence of subgaussian properties
(see, e.g.,~\cite[Proposition~2.6.1]{Vershynin_2026}), this implies
\[
\|f(X) - \E f(X)\|_{\psi_2}
\le C \sqrt{\kappa(\Sigma)}\|b\|_2,
\]
where $C>0$ is an absolute constant.
\end{proof}

The coordinatewise result from the first version is an immediate
consequence.

\begin{corollary}\label{thm:bounded_spectral}
Let $X\sim\mathcal N(\mu,\Sigma)$ be an $n$-dimensional Gaussian vector with non-singular covariance matrix $\Sigma$. Let $\phi:\R\to\R$ be a measurable function with $\|\phi\|_\infty\le 1$, and define
$Y=\phi(X)$ coordinatewise. Then
\[
\|Y-\E Y\|_{\psi_2}\le C\sqrt{\kappa(\Sigma)},
\]
where $C>0$ is an absolute constant.
\end{corollary}

This follows immediately: for $v\in S^{n-1}$, the function
$f_v(x)=\sum_i v_i\phi(x_i)$ has bounded-differences constants
$b_i\le2|v_i|$, and hence $\|b\|_2\le2$. Theorem~\ref{thm:bounded_differences}
and the definition of the subgaussian norm of a random vector give the claim
after taking the supremum over $v$ and absorbing the factor $2$ into $C$.

\begin{remark}[Extension to Continuous Functions]
If $\phi$ is $\gamma$-H\"older continuous instead of globally bounded, the
same decomposing trick still works, and the subgaussian norm becomes bounded
by
\[
O\!\left(\sqrt{\lambda_{\max}(\Sigma)}\,
\lambda_{\min}(\Sigma)^{\frac{\gamma-1}{2}}\right).
\]
See the
\href{https://gemini.google.com/share/cefbb71da6d2}{conversation history}
for the derivation.

Furthermore, the well-conditioning of $\Sigma$ is necessary for a uniform bound. A natural counterexample is the rank-one matrix $\Sigma=\mathbf 1\mathbf 1^\top$, where $\mathbf 1=(1,\ldots,1)^\top\in\R^n$, which yields $\|\sgn(X)\|_{\psi_2}\ge c\sqrt n$ for an absolute constant $c>0$.
\end{remark}

\begin{remark}[Simultaneous work of Roth]
Simultaneously and independently of our work,
Roth~\cite{roth2026mcdiarmid} obtained a bounded-differences inequality under dependence by a different method. His result applies to distributions more general than Gaussian distributions, including strongly log-concave and log-smooth measures. Our argument is restricted to Gaussian inputs, but it is
elementary and short.
\end{remark}

\section{Application: Sign-Quantized Square Linear Maps}
\label{sec:application}

We now apply Corollary~\ref{thm:bounded_spectral} to resolve the question posed by Simone Bombari.

\medskip
\begin{corollary}\label{cor:bombari_resolution}
Let $W\in\R^{n\times n}$ be a random matrix with i.i.d. standard normal
entries $W_{ij}\sim\mathcal N(0,1)$. Let $x\sim\mathcal N(0,I_n)$ be an
$n$-dimensional standard normal random vector. Define
$Y=\sgn(Wx)\in\{-1,1\}^n$ by applying the sign function to each coordinate
of $Wx$. Then
\[
\E\norm{Y}_{\psi_2\mid W}\le C,
\]
where $C>0$ is an absolute constant. Here, the conditional norm is with
respect to $x$ for fixed $W$, and the expectation is with respect to $W$.
\end{corollary}

\begin{proof}
The case $n=1$ is immediate, so assume $n\ge 2$.
Directly applying Corollary~\ref{thm:bounded_spectral} to $W$ fails because
$\Sigma=WW^\top$ typically has large condition number.

To fix this, we split $W$ into two row blocks
$W_1\in\R^{m\times n}$ and $W_2\in\R^{(n-m)\times n}$, where
$m=\lfloor n/2\rfloor$:
\begin{equation}\label{equ:partition_trick}
W=\begin{pmatrix}W_1\\W_2\end{pmatrix}.
\end{equation}
This induces the split
\[
Y^{(1)}=\sgn(W_1x),
\qquad
Y^{(2)}=\sgn(W_2x).
\]
The triangle inequality gives
\[
\norm{Y}_{\psi_2\mid W}
\le
\norm{Y^{(1)}}_{\psi_2\mid W_1}
+\norm{Y^{(2)}}_{\psi_2\mid W_2}.
\]
Both blocks have at most $2n/3$ rows. The same argument applies to both, so
we treat $W_1$.

Since $m/n\le 1/2$, the rectangular block $W_1$ is well conditioned with
high probability. Specifically, its covariance
$\Sigma_1=W_1W_1^\top$ satisfies
\[
nc_1I\preceq\Sigma_1\preceq nc_2I
\]
with probability at least $1-c_3e^{-c_4n}$ for universal positive constants
$c_1,c_2,c_3,c_4$; see
\cite[Corollary~7.3.2 and Exercise~7.13]{Vershynin_2026}. Let
$\mathcal E=\{\kappa(\Sigma_1)\le c_2/c_1\}$.

Since $W_1x$ is centered Gaussian,
$\E[Y^{(1)}\mid W_1]=0$. On $\mathcal E$,
Corollary~\ref{thm:bounded_spectral} gives
\begin{equation}\label{equ:Y_1_bound}
\norm{Y^{(1)}}_{\psi_2\mid W_1}
\le C_0\sqrt{\kappa(\Sigma_1)}
\le C_0\sqrt{\frac{c_2}{c_1}}
\eqqcolon C_1.
\end{equation}
On $\mathcal E^c$, the deterministic bound
$|\langle v,Y^{(1)}\rangle|\le\sqrt m\le\sqrt n$ gives
$\norm{Y^{(1)}}_{\psi_2\mid W_1}\le c_5\sqrt n$ for an absolute constant
$c_5>0$. Therefore
\[
\E\norm{Y^{(1)}}_{\psi_2\mid W_1}
\le C_1+c_5c_3\sqrt n\,e^{-c_4n}
\le C_2.
\]
Here $C_2>0$ is an absolute constant.
The same estimate for $W_2$ and the triangle inequality complete the proof.
\end{proof}

\begin{remark}[Role of the AI system]
Gemini initially returned the covariance-splitting argument for the sign
function and extended it to bounded coordinate maps. We later learned that
the sign case and the same proof strategy appeared in Barber and
Kolar~\cite[Lemma~4.5]{MR3852657}; Gemini did not identify that source. The
bounded-differences formulation above follows the same argument. The
row-partitioning step in \eqref{equ:partition_trick} was added by the authors.
\end{remark}
\vspace{-0.5cm}
\section*{Acknowledgments}
The authors thank Simone Bombari for posing the question that motivated this
work and for pointing out its connection with the bounded-differences
inequality. We thank Valentin Roth for suggesting that we state our argument
in bounded-differences form and for sharing his simultaneous and independent work on
bounded-differences inequalities under dependence. Google Gemini 3.5 Flash
was used during the initial exploration, as described above. A reproduced
interactive dialogue and derivation transcripts are archived at
\url{https://gemini.google.com/share/cefbb71da6d2}. The authors are supported
by NSF Grant DMS 2451011 and U.S. Air Force Grant FA9550-25-1-0294.
\vspace{-0.5cm}
\bibliographystyle{alpha}

\end{document}